\documentclass[11pt,a4paper]{article}
\usepackage{amsmath, amscd}
\usepackage{amssymb}
\usepackage{mathrsfs, calrsfs}
\usepackage{graphicx, subcaption}
\usepackage{amsthm}
\usepackage{textcomp}
\usepackage{epsfig} 
\usepackage{epstopdf}

\usepackage{array}
\usepackage{setspace}

\newtheorem{theorem}{Theorem}[section]

\newtheorem{lemma}{Lemma}[section]
\newtheorem{example}{Example}[section]

\newcommand{\ioa}{\int_0^{a_{\dag}}}
\newcommand{\dt}{\Delta t}
\newcommand{\ad}{a_{\dag}}

\title{\Large A convergent numerical scheme to a parabolic equation with a nonlocal boundary condition}
\date{}
\author{Bhargav Kumar Kakumani
\footnote{Department of Mathematics, BITS-Pilani, Hyderabad Campus, 500078, India.
	 E-mail:\textit{bhargav@hyderabad.bits-pilani.ac.in}} 
 \footnote{Corresponding author }  
\and Suman Kumar Tumuluri
\footnote{School of Mathematics and Statistics, University of Hyderabad, Hyderabad, 500046, India.
	 E-mail:\textit{suman.hcu@gmail.com}}
 }

\begin{document}
 \maketitle
\pagestyle{plain}
\pagenumbering{arabic}

\begin{abstract}
	In this paper, a numerical scheme for a nonlinear McKendrick--von Foerster equation with diffusion in age (MV-D) with the Dirichlet boundary condition 
	is proposed. The main idea to derive the scheme is to use the discretization based on the method of characteristics to the convection part, and the finite difference method to the rest of the terms. The nonlocal terms are dealt with the quadrature methods. As a result, an implicit scheme is obtained for the boundary value problem under consideration. The consistency, and the convergence of the proposed numerical scheme is established. Moreover, numerical simulations are presented to validate the theoretical results.
\end{abstract}




\section{Introduction}
The structured population models distinguish individuals from one another based on characteristics such as size, age, and so on. Moreover the birth, and the death rates also depend on these characteristics. Structured models are broadly classified as age structured models (\cite{chipot,cushing,farkas,ian1,michel1,michel2,perthame,per_suman}), gender structured models (\cite{calsina,ian2005,thieme}), size structured models (\cite{per.1,farkas1,farkas2,diekmann}), and maturity structured models (\cite{dyson1,dyson2,per_michler,rubi}).\\

\noindent 
Many size structured models describing the behavior of the cell populations have been introduced and studied (\cite{per_ry} and the references therein). There are four important phases in the mammalian cell cycle, they are $G_1$ - phase, $S$ - phase, $G_2$ - phase and $M$ - phase. In $S$-phase and $M$-phase there will be DNA replication and in $M$-phase cell division occurs. The authors in \cite{basse} considered a size structured model and used DNA content as a measure of the genetic term ‘cell size’, because the vital transitions (i.e., from $G_1$-phase to $S$-phase, from $M$-phase to cell division, etc.) are accompanied in DNA content. These changes in the DNA content can be measured by flow cytometry, but sometimes the cells are not evenly illuminated when they pass through the flow cytometer. Hence, in order to account for these changes in the cell cycle model, the authors introduced the diffusive term. \\

In this paper, we consider the McKendrick–von Foerster equation with diffusion in the age variable with the Dirichlet boundary condition. Let $u(t,x)$ denote the population density with age $x$ at time $t$. Let $d$ and $ B$ denote the mortality rate and the fertility rate, respectively. Let $\psi$ and $S$ denote the competition weight and weighted population, respectively. Let $g$ denote a non-negative $C^1$ function defined on $[0,\infty)$ that satisfies some assumptions which
will be described later. \\
With these notations, we consider the following nonlinear MV-D in\\ $Q:=(0,T)\times (0,\ad)$, 
\begin{equation}
\label{mvd-1}
\left\{ \begin{array}{ll}
u_t(t,x) + u_x(t,x) +d(x,S(t))u(t,x) = \epsilon u_{xx}(t,x), \hspace{0.1cm} (t,x)\in Q,\medskip \\
u(t,0) = g\Big(\displaystyle \ioa B(x)u(t,x)dx\Big), \hspace{0.2cm} t\in (0,T),\medskip \\
u(t,\ad)=0, \hspace{0.1cm} t\in (0,T),\medskip \\
u(0,x) = u_0(x),\hspace{0.1cm} x\in (0,\ad),
\end {array} \right.
\end{equation}
where
\begin{equation}
\label{mvd-2}
S(t)=\ioa \psi(x)u(t,x)dx.
\end{equation}

\noindent We assume that $u\in W^{3,\infty}(Q)$ and the compatibility criterion, $i.e.,$ $$u_0(0)=g\Big(\displaystyle \ioa B(x)u_0(x)dx\Big).$$

\noindent Different variants of MV-D equations attracted the attention of many mathematicians. In \cite{tarik}, the authors considered the linear MV-D equation with Robin boundary data and studied the existence, uniqueness, and positivity of the solution. In addition, they proved the exponential decay of the solution to MV-D equation for large time to the steady state. The existence, uniqueness, and the long time behavior of the solution for a particular type of non-linear MV-D were proved in \cite{tarik-michel}. The authors in \cite{bharu-suman,michel-bharu} considered the nonlinear MV-D with Robin boundary data and proved the existence and uniqueness of its solution using the fixed point arguments. Moreover, the authors studied the asymptotic behavior of the solution towards its steady state using the notion of the General Relative Entropy. There is a lot of  literature available in the numerical study of McKendrick--von Foerster equation (MV) with Dirichlet boundary data and there are very few papers available related to MV-D with Robin boundary data (\cite{bharu-suman1}).  In \cite{lopez1,lopez2}, the authors developed the concept of stability to certain class of nonlinear problems. They proved that for a smooth discretization, the stability was equivalent to the stability of its linearization around the theoretical solution. This concept of stability works for MV type equation with Dirichlet boundary data. In \cite{abia,lopez3,lopez4,lopez5}, the authors considered age structured model with Dirichlet boundary data and developed a new numerical method to approximate the solution to the equation that they considered. The analysis involved theory of discretizations based on the notion of stability thresholds. In \cite{ian2,ian3}, Iannelli $et\ al.$ presented finite difference methods to find the numerical solution to the linear Lotka-McKendrick equation. Furthermore, they showed that the scheme was convergent in the maximum norm and also discussed the stability of the scheme.  In \cite{kim}, Kim $et\ al.$ used the collocation method along the characteristics to approximate the solution to the nonlinear MV equation with Dirichlet boundary data. Their scheme indeed converges and it is of the fourth order accuracy in the supremum norm. In \cite{sun}, the authors presented a finite difference numerical scheme to a class of nonlinear nonlocal equations with Robin type boundary conditions. By the method of reduction of order, the authors proved that the numerical scheme is second order accurate.\\
\noindent
In \cite{ben1}, the authors considered an age structured alcoholism model. They investigated the global behavior of the solution of the model using the basic reproduction number, and analysed the nature of the nontrivial equilibrium when the basic reproduction number is greater than one. The authors in \cite{ben2,ben3} studied a nonlinear epidemic model and obtained the existence result for the solutions. Using the basic reproduction number they provided a necessary and sufficient condition for global asymptotic stability of the free-equilibrium. In \cite{ben4}, Bentout $et\ al.$  presented a mathematical model that predicts the spread of the pandemic  COVID-19. By considering the age as a factor of the progress and severity of the disease, the authors proposed an age structured model. Using this model they predicted the size of the epidemic in the USA, the UAE, and Algeria.\\
\noindent 
There are many more mathematical models which are age structured. For instance, the model of stress erythropoiesis (see \cite{jcl1} and references therein), the model which describes cell dwarfism \cite{jcl2}.  Here, the authors proposed a numerical method to obtain an approximation to its solution and established the convergence of the numerical method. The authors in \cite{jcl3} considered a size structured cell population model. They proposed a second order numerical method to approximate its solution and showed the convergence of the proposed scheme.    \\
\noindent 
In this paper, we present a numerical method to find an approximate solution to equation \eqref{mvd-1}--\eqref{mvd-2}. The present numerical method is based on the method of characteristics (MOC) (see \cite{jim}). It is easy to compute the solution to nonlinear boundary value problem \eqref{mvd-1}--\eqref{mvd-2} using the method that we propose because there is no need to solve any nonlinear equation at the discrete level. On the other hand, we take the central difference approximation for the diffusion term at the $n$-th stage so that we can have the convergence of the scheme.  \\

\noindent 
This paper is organised as follows. In Section $2$, we propose a numerical scheme to equations (\ref{mvd-1})--(\ref{mvd-2}) and prove that the numerical solution obtained from the proposed scheme converges to the solution to the MV-D equation. Numerical results are presented in Section $3$.

\section{The finite difference scheme}
\label{num_moc}
Let $h$ and  $\Delta t$ denote the uniform step size of age variable and time variable, respectively.  Moreover, we denote\medskip\\
\centerline{$x_j=jh,\ t^n=n\Delta t,\ u_j^n=u(t^n,x_j),\ j=1, \ldots, M,$  $n=1, \ldots, N,$} \\
and \\
\centerline{$\bar{x}_j=x_j-\Delta t,\ \bar{u}_j^n=u(t^n,\bar{x}_j),\ j=0,1,\ldots ,M$ and $n=0,1,\ldots,N$.} \\

\noindent 
Observe that (\ref{mvd-1})--(\ref{mvd-2})
reduces to a hyperbolic equation when $\epsilon=0$. Thus we use the method of characteristics to discretize the convection part (namely, $u_t+u_x$). The standard central difference approximation is used to deal with the diffusion part (namely $u_{xx}$). We discretize the hyperbolic part using the backward finite difference $i.e.,$
\[
u_t(t,x)+u_x(t,x) \sim \frac{1}{\dt}\big(u(t,x)-u(t-\dt,x-\dt)\big).
\]
Let $U_i^n $ denote the approximation of $u(t^n,x_i)$ at every grid $(t^n,x_j)$. Assume that $q_j$'s are the Newton-Cote numbers such that the approximation of the integral is\\
\[ \int_0^{\ad}B(x)u(t,x)dx \sim \sum\limits_{j=1}^M h q_j B(x_j)U_{j}^{n}, \]
which is of the order $h^3$. With this notation, we now discretize (\ref{mvd-1})--(\ref{mvd-2}) to get
\begin{equation}
\label{num-sch}
\left\{ \begin{array}{ll}
U_j^{n}=\bar{U}_j^{n-1} - \dt d(x_j,S^{n-1})U_{j}^{n-1} + \frac{\epsilon \dt}{h^2}\big(U_{j+1}^{n}-2U_{j}^{n}+U_{j-1}^{n}\big),  \medskip \\
\hfill \hspace{0.2cm} i=1,2,\ldots, M-1, \hspace{0.1cm} n=1,2,\ldots, N,\medskip \\
U_0^{n}=g\Big(\sum\limits_{j=1}^M h q_j B(x_j)U_{j}^{n-1}\Big), \hspace{0.2cm} n=1,\ldots, N,\medskip \\
U_M^{n}=0, \hspace{0.2cm} n=1,\ldots, N,\medskip \\
U_j^{0}=u_0(x_j), \hspace{0.2cm} j=1,\ldots, M,\medskip \\
S^n=\sum\limits_{i=1}^M hq_j \psi(x_j)U_{j}^{n},
\end {array} \right.
\end{equation}

\noindent where $\bar{U}_j^{n-1}$ is the evaluation of the linear interpolation value taken between $U_j^{n-1}$ and $U_{j-1}^{n-1}$ at $\bar{x}$.

\noindent 
We now present a consistency result which holds whenever the mortality rate satisfies the Lipschitz condition given in the following lemma.

\begin{lemma}
	\label{consistency}
	Assume that there exists $L>0$ such that \\
	\begin{equation}
	\label{d_lips}
	|d(x,S_1)-d(x,S_2)|\leq L|S_1-S_2|,\ S_1,S_2\geq 0,\ x\in (0,\ad).
	\end{equation}
	Then following local truncation results hold:
	\begin{enumerate}
		\item $u_t(t^n,x_j)+u_x(t^n,x_j) = \displaystyle \frac{u^n_j - \bar{u}_j^{n-1}}{\dt} + O(\dt),$ \medskip 
		\item $d(x_j,S(t^n)) = d(x_j,\sum\limits_{i=1}^M hq_j \psi_j u_{j}^{n})+ O(h^3),$ \medskip
		\item $u_{xx}(t^n,x_j)= \displaystyle \frac{u_{j+1}^{n}-2u_{j}^{n}+u_{j-1}^{n}}{h^2} + O(h^2).$
	\end{enumerate}
\end{lemma}

\begin{proof} Let $u$ be the solution to (\ref{mvd-1})--(\ref{mvd-2}). Let $\tau:=\tau(x)$ be the characteristic direction associated with the operator $\partial_t+\partial_x$. Then we have
	\[
	\frac{\partial u}{\partial \tau} = \frac{1}{\sqrt{2}}\frac{\partial u}{\partial t} + \frac{1}{\sqrt{2}}\frac{\partial u}{\partial x}.
	\] 
	We approximate the characteristic derivative as 
	\[
	\sqrt{2} \frac{\partial u(t^n,x_j)}{\partial \tau} = \sqrt{2} \frac{u(t^n,x_j) - u(t^{n-1},\bar{x}_j)}{\big[(x_j-\bar{x}_j)^2+(t^{n}-t^{n-1})^2\big]^{\frac{1}{2}}} + O(\dt) = \displaystyle \frac{u^n_j - \bar{u}_j^{n-1}}{\dt} + O(\dt).
	\]
	In order to prove $(ii)$, we use that $d$ satisfies Lipschitz condition (\ref{d_lips}). Since $q_j$'s are chosen such that \\
	\[
	|S(t^n)-h\sum\limits_{j=1}^Mq_j\psi_ju_j^n| = O(h^3),
	\]
	it follows that \\
	\[ |d(x_j,S(t^n)) - d(x_j,\sum\limits_{i=1}^M hq_j \psi_j u_{j}^{n})|= O(h^3).
	\]
	Note that $(iii)$ is straightforward from the formal Taylor series expansion of $u$ about $(t^n,x_j)$.
\end{proof}

\noindent 
We now state and prove the main result of this paper, $i.e.,$ the convergence of the numerical solution $U_j^n$ to the solution $u$ to (\ref{mvd-1})--(\ref{mvd-2}) at the grid points. \\

\begin{theorem}
	\label{convergence_thm} Assume the hypothesis of Lemma \ref{consistency}. Moreover assume that
	\begin{equation}
	\label{cvg_assumption}
	\|g'\|_{\infty}\|B\|_{\infty}\ad<1,\ \frac{\dt}{h^2}\leq \frac{1}{2}.
	\end{equation}
	Fix $T>0$ and let $N\in\mathbb{N},\Delta t>0$ be such that $N\Delta t=T.$
	Then the solution to numerical scheme (\ref{num-sch}) converges to the solution to (\ref{mvd-1})--(\ref{mvd-2}) as $h\rightarrow 0,\ \Delta t\rightarrow 0$. Moreover, there exists a positive constant $\tilde C,$ independent of $h,\Delta t$, such that \medskip\\
	\centerline{$ \max\limits_{\substack{{0\leq j\leq M}\\ {0\leq n\leq N}}}|u_j^n-U_j|\leq \tilde C(h+\Delta t).$}
	
\end{theorem}

\begin{proof}
	
	\noindent In view of Lemma \ref{consistency}, we get the consistency of the scheme as follows:
	\begin{equation}
	\label{consistency-sch}
	\left\{ \begin{array}{ll}
	\displaystyle\frac{u_j^{n}-\bar{u}_j^{n-1}}{\dt} +d(x_j,S(t^{n-1}))u_{j}^{n-1}  =\frac{\epsilon}{h^2}\big(u_{j+1}^{n}-2u_{j}^{n}+u_{j-1}^{n}\big) + e^n_j,  \medskip \\
	\hfill \hspace{0.2cm} j=1,2,\ldots,M-1, \hspace{0.1cm} n=1,2,\ldots,N,\medskip \\
	u_0^{n}=g\Big(\sum\limits_{j=1}^M h q_j B(x_j)u_{j}^{n-1}\Big)+\tilde{e}_j^n, \hspace{0.2cm} n=1,\ldots,N,\medskip \\
	u_M^{n}=0, \hspace{0.2cm} n=1,\ldots,N,\medskip \\
	u_j^{0}=u_0(x_j), \hspace{0.2cm} j=1,\ldots,M,\medskip \\
	S(t^n)=\sum\limits_{i=1}^M hq_j \psi(x_j)u_{j}^{n}+\tilde{\tilde{e}}_j,\ n=1,\ldots,N,
	\end {array} \right.
	\end{equation}
	\noindent where $e_j^n = O(h+\dt),\ \tilde{e}_j^n=O(h^3)$ and $\tilde{\tilde{e}}_j=O(h^3)$.  \\
	
	\noindent Let $\rho_j^n=u_j^n-U_j^n,\ \bar{\rho}_j^n=\bar{u}_j^{n}-\bar{U}_j^n$ and $S^n=\sum\limits_{i=1}^M hq_j \psi(x_j)U_{j}^{n}$. From (\ref{num-sch})--(\ref{consistency-sch}), it follows that 
	
	\begin{equation}
	\label{num-sch1}
	\left\{ \begin{array}{ll}
	\displaystyle\frac{\rho_j^{n}-\bar{\rho}_j^{n-1}}{\dt} +d(x_j,S(t^{n-1}))u_{j}^{n-1} - d(x_j,S^{n-1})U_{j}^{n-1} \medskip \\
	\hspace{2cm} =\displaystyle\frac{\epsilon}{h^2}\big(\rho_{j+1}^{n}-2\rho_{j}^{n}+\rho_{j-1}^{n}\big) + e^n_j,  \medskip \\
	\hfill \hspace{0.2cm} j=1,2,\ldots,M-1, \hspace{0.1cm} n=1,2,\dots,N,\medskip \\
	\rho_0^{n}=g\Big(\sum\limits_{j=1}^M h q_j B(x_j)u_{j}^{n-1}\Big)-g\Big(\sum\limits_{j=1}^M h q_j B(x_j)U_{j}^{n-1}\Big)+\tilde{e}_j^n, \medskip \\
	\hfill 
	\hspace{0.2cm} n=1,\dots,N,\medskip \\
	\rho_M^{n}=0, \hspace{0.2cm} n=1,\dots,N,\medskip \\
	\rho_j^{0}=0, \hspace{0.2cm} j=1,\dots,M.
	\end {array} \right.
	\end{equation}
	\noindent We first estimate 
	\begin{equation*}
	\begin{array}{ll}
	|S(t^n)-S^{n}| &\leq\displaystyle \Big|\ioa \psi(x)u(t^n,x)dx - \sum\limits_{i=1}^M hq_j \psi(x_j)U_{j}^{n}\Big| \medskip \\
	& \leq \displaystyle \Big|\ioa \psi(x)u(t^n,x)dx - \sum\limits_{i=1}^M hq_j \psi(x_j)u_{j}^{n}\Big| + \Big| \sum\limits_{i=1}^M hq_j \psi(x_j)\rho_{j}^{n}\Big| \medskip \\
	& \leq O(h^3)+ \|\psi\|_{\infty}\ad \max\limits_{1\leq j\leq M} |\rho_j^n|,\ n\in \mathbb{N}.
	\end{array}
	\end{equation*}
	\noindent Furthermore, using the above estimate, we obtain\medskip\\
	$|d(x_j,S(t^{n-1}))u_{j}^{n-1} - d(x_j,S^{n-1})U_{j}^{n-1}|$
	\begin{equation}
	\label{d_estimate}
	\begin{array}{ll}
	& \leq |d(x_j,S^{n-1})\rho_{j}^{n-1}|
	+ |d(x_j,S(t^{n-1}) - d(x_j, S^{n-1})| |u_{j}^{n-1}|\medskip \\
	& \leq \big(\|d\|_{\infty}+L\|u\|_{\infty}\|\psi\|_{\infty}\ad\big)\max\limits_{1\leq j\leq M} |\rho_j^{n-1}|+O(h^3).
	\end{array}
	\end{equation}
	
	\noindent We now estimate the boundary term in (\ref{num-sch1}). In order to do that, we consider 
	\begin{equation}
	\label{bdry}
	\begin{array}{ll}
	|\rho_0^{n}|& =|g\Big(\sum\limits_{j=1}^M h q_j B(x_j)u_{j}^{n-1}\Big)-g\Big(\sum\limits_{j=1}^M h q_j B(x_j)U_{j}^{n-1}\Big)+\tilde{e}_j^n| \medskip \\
	& \leq |g'(\xi_j)| |\sum\limits_{j=1}^M h q_j B(x_j)\rho_{j}^{n-1}| + O(h^3) \medskip \\
	& \leq  \|g'\|_{\infty} \|B\|_{\infty}\ad \max\limits_{1\leq j\leq M} |\rho_j^{n-1}|+ O(h^3).
	\end{array}
	\end{equation}
	\noindent Let $I$ and $I_1$ denote the identity operator and the linear interpolation operator, respectively. We observe that 
	\[
	\bar{\rho}_j^{n-1}=I_1(\rho^{n-1})(\bar{x}_j) + (I-I_1)(u^{n-1})(\bar{x}_j),
	\]
	or  \[
	|\bar{\rho}_j^{n-1}| \leq \max\limits_{0\leq j\leq M} |\rho_j^{n-1}| + |(I-I_1)(u^{n-1})(\bar{x}_j)|.
	\]
	In view of the Peano kernel theorem, we have 
	\begin{equation}
	\label{peano_application}
	\max\limits_{1\leq j\leq M}|\bar{\rho}_j^{n-1}| \leq \max\limits_{0\leq j\leq M} |\rho_j^{n-1}| + \|u^{n-1}\|_{2,\infty}h\dt.
	\end{equation}
	From the first equation of (\ref{num-sch1}) and estimates (\ref{d_estimate})--(\ref{peano_application}), it follows that 
	\begin{equation}
	\label{bdry_c}
	\begin{array}{ll}
	\max\limits_{1\leq j\leq M}|\rho_j^{n}| &\leq  \max\limits_{1\leq j\leq M}|\bar{\rho}_j^{n-1}| + C\dt \max\limits_{1\leq j\leq M}|\rho_j^{n-1}| + O(h+\dt)\dt \medskip \\
	&\leq (1+C\dt)\max\limits_{0\leq j\leq M}|\rho_j^{n-1}|   + O(h+\dt)\dt.
	\end{array}
	\end{equation}
	
	\noindent From (\ref{bdry}) and (\ref{bdry_c}), we get
	\begin{equation*}
	\begin{array}{ll}
	|\rho_0^{n}|& \leq M\Big\{h^3 + (h+\dt)\dt \big[ 1 + (1+C\dt) + \cdots + (1+C\dt)^{N-2} \big] \Big\} \medskip \\
	& \leq M \Big\{ h^3 + (h+\dt)\dt N e^{CT} \Big\}\medskip \\
	& \leq  M \Big\{ h^3 + (h+\dt)T e^{CT}\Big\},
	\end{array}
	\end{equation*}
	and 
	\begin{equation*}
	\begin{array}{ll}
	\max\limits_{1\leq j\leq M}|\rho_j^{n}| &\leq  M(h+\dt)\dt \big[  1 + (1+C\dt) + \cdots + (1+C\dt)^{N-1} \big] \medskip \\
	& \leq  M(h+\dt)\dt N e^{CT}\medskip \\
	& \leq  M(h+\dt)T e^{CT}.
	\end{array}
	\end{equation*}
	
	\noindent Hence we find that
	\begin{equation}
	\label{convergence_result}
	\max\limits_{\substack{{0\leq j\leq M}\\ {0\leq n\leq N}}}
	|\rho_j^{n}| = O(h+\dt)\ {\rm as}\ h\rightarrow 0,\ \Delta t\rightarrow 0.
	\end{equation}
	This completes the proof of the promised result.
\end{proof}

\section{Numerical simulations}
\label{num_numerics}
In this section, we present some examples in which we have performed numerical simulations. In the first two examples, we have taken the linear version of \eqref{mvd-1}--\eqref{mvd-2}, $i.e.,\ d(x,S)=d(x)$, $g(x)=x,$ and $B$ is chosen such that the hypotheses of Theorem \ref{convergence_thm} hold. On the other hand, in the last two examples we have considered nonlinear equations.

\begin{example} \label{eg01_lnr}
	We first assume that $\ad=1$. Moreover assume that the vital rates $d,B,$ and the initial data $u_0$ are
	given by\medskip\\
	\centerline{$\displaystyle d(x,S)=\frac{3e^{-x}-e^{-1}}{e^{-x}-e^{-1}},\ B(x)=1+\frac{e^{-1}}{1-2e^{-1}}, \epsilon = 1.0, \ u_0(x)=e^{-x}-e^{-1}.$}
\end{example}
\noindent In this case, it is easy to verify that the solution to (\ref{mvd-1})--(\ref{mvd-2}) is given by $$u(t,x)=e^{-t}(e^{-x}-e^{-1}),\ \ t>0,\ x\in (0,1).$$ For the numerical solution, we have taken $h=0.0025, \Delta t = 3.125 \rm{x}10^{-6}$. In Figure \ref{eg1_linear}(a), we have shown the analytical solution and the numerical solution when $N=16000,$ and  $64000$. The absolute error is shown in Figure \ref{eg1_linear}(b). From these figures, it is clear that the numerical solution is very close to the analytical solution. In Figure \ref{eg1_linear_comp} we have presented the numerical plots for different values of $h$ and $\Delta t$ at $t=0.1$ and also compared the same with the corresponding analytical solution. It is evident from Figure \ref{eg1_linear_comp} that as $h$ and $\Delta t$ tend to zero, the numerical solution converges to the analytical solution.

\begin{figure}[h]
	\includegraphics[width=1.2\textwidth]{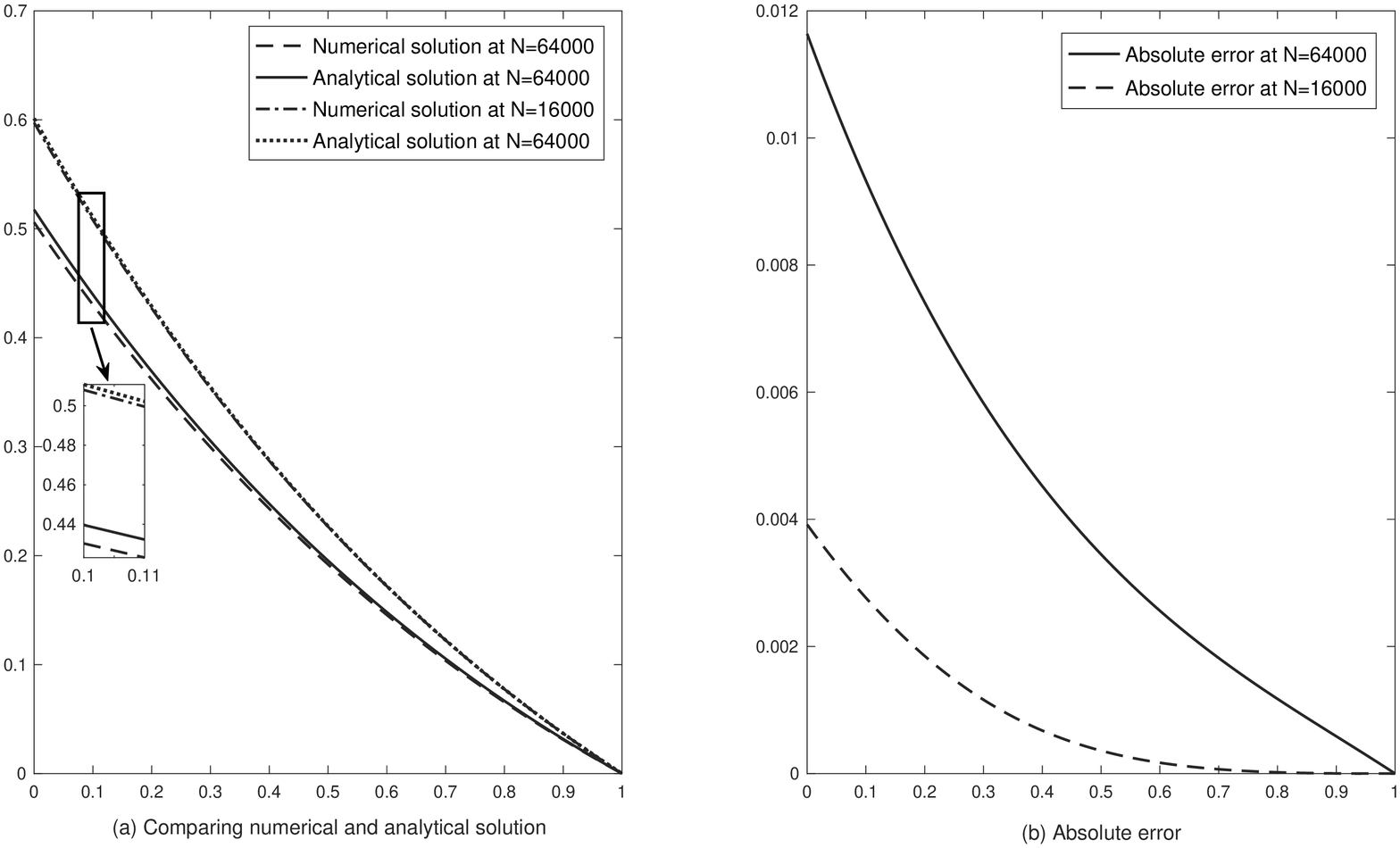}
	\caption{(a) The numerical solution and the exact solution to (\ref{mvd-1})--(\ref{mvd-2}) with $N=16000,\ 64000$ with vital rates given in Example \ref{eg01_lnr}; (b) The absolute difference between $u(N\Delta t, x_i)$ and $U^{N\Delta t}_{x_i}$ when $N=16000$ and $64000$.}
	\label{eg1_linear}
	\includegraphics[width=1.1\textwidth]{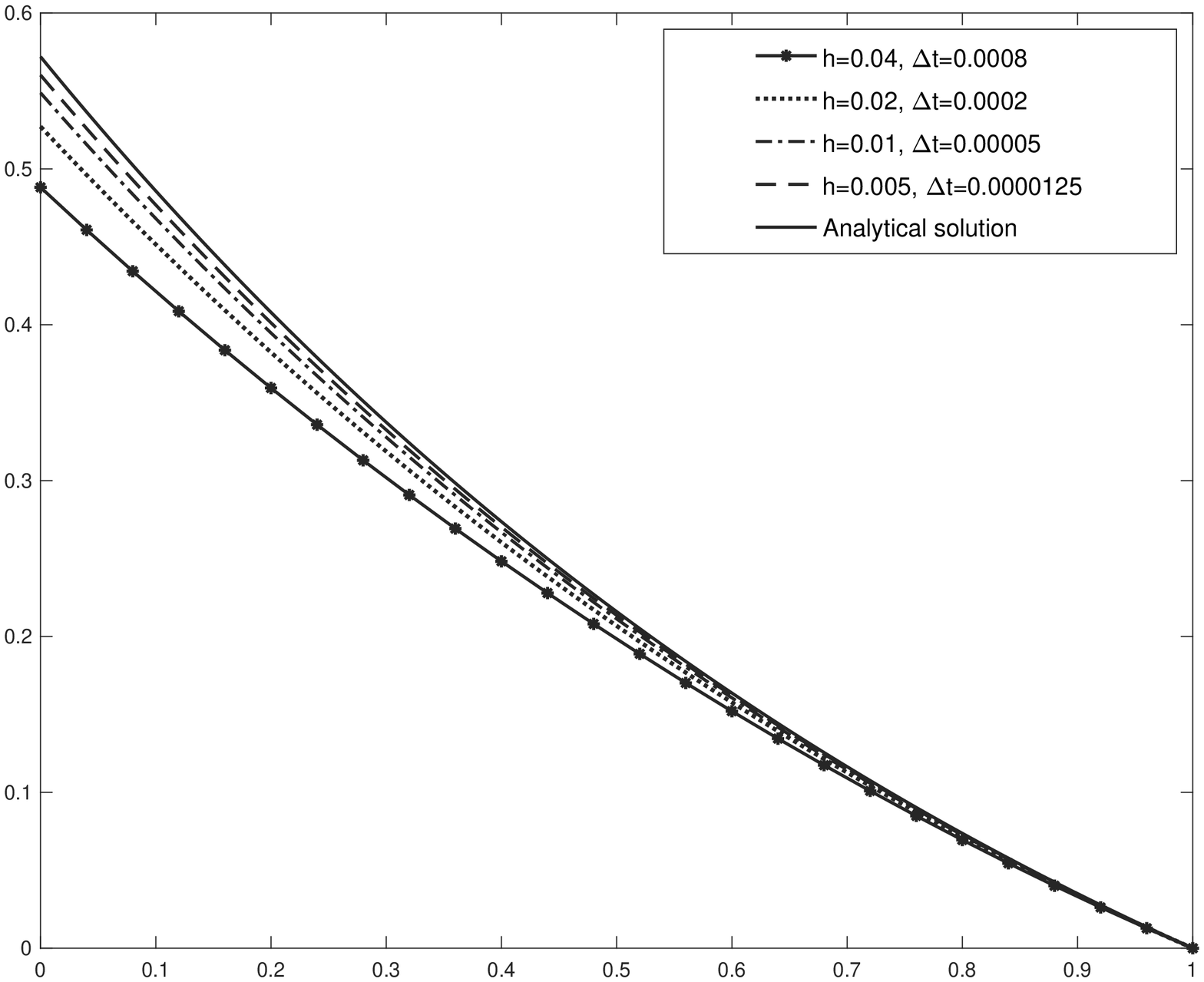}
	\caption{Numerical solution with different values of $h,\ \Delta t$ and the corresponding analytical solution to (\ref{mvd-1})--(\ref{mvd-2})  at $t=0.1$ with vital rates given in Example \ref{eg01_lnr}.}
	\label{eg1_linear_comp}
\end{figure}

\begin{example} \label{eg02_lnr}
	In this example also, we take $\ad =1$. The vital rates $d,B,$ and the initial data $u_0$ are 
	given by\medskip\\
	\centerline{$\displaystyle d(x,S)=3\Big(1+ \frac{1}{1 - x} \Big) ,\ B(x)=\frac{4}{1+e^{-2}} e^{-x}, \epsilon = 1.0, \ u_0(x)=e^{-x}(1 -x).$}
	
\end{example}
\noindent It is straightforward to verify that the solution to (\ref{mvd-1})--(\ref{mvd-2}) in this example is given by $$u(t,x)=e^{-t}e^{-x}(1 -x),\ \ t>0,\ x\in (0,1).$$ For the numerical computations, we have taken $h=0.002, \Delta t = 2.0\rm{x}10^{-6}$. In Figure \ref{eg2_linear}(a), we plot the analytical solution and the numerical solution when $N=9375,$ and $75000$. The absolute error is presented in Figure \ref{eg2_linear}(b). This example also validates that the numerical method proposed in this article gives an approximate solution which is in a good agreement with the exact solution. In Figure \ref{eg2_linear_comp}  the numerical solutions for different values of $h$ and $\Delta t$ at $t=0.15$ are shown, Moreover, the corresponding analytical solutions are also presented in the same figure for the comparison. We conclude from Figure \ref{eg2_linear_comp} that as $h$ and $\Delta t$ go to zero, the numerical solution converges to the analytical solution.

\begin{figure}[h]
	\includegraphics[width=1.2\textwidth]{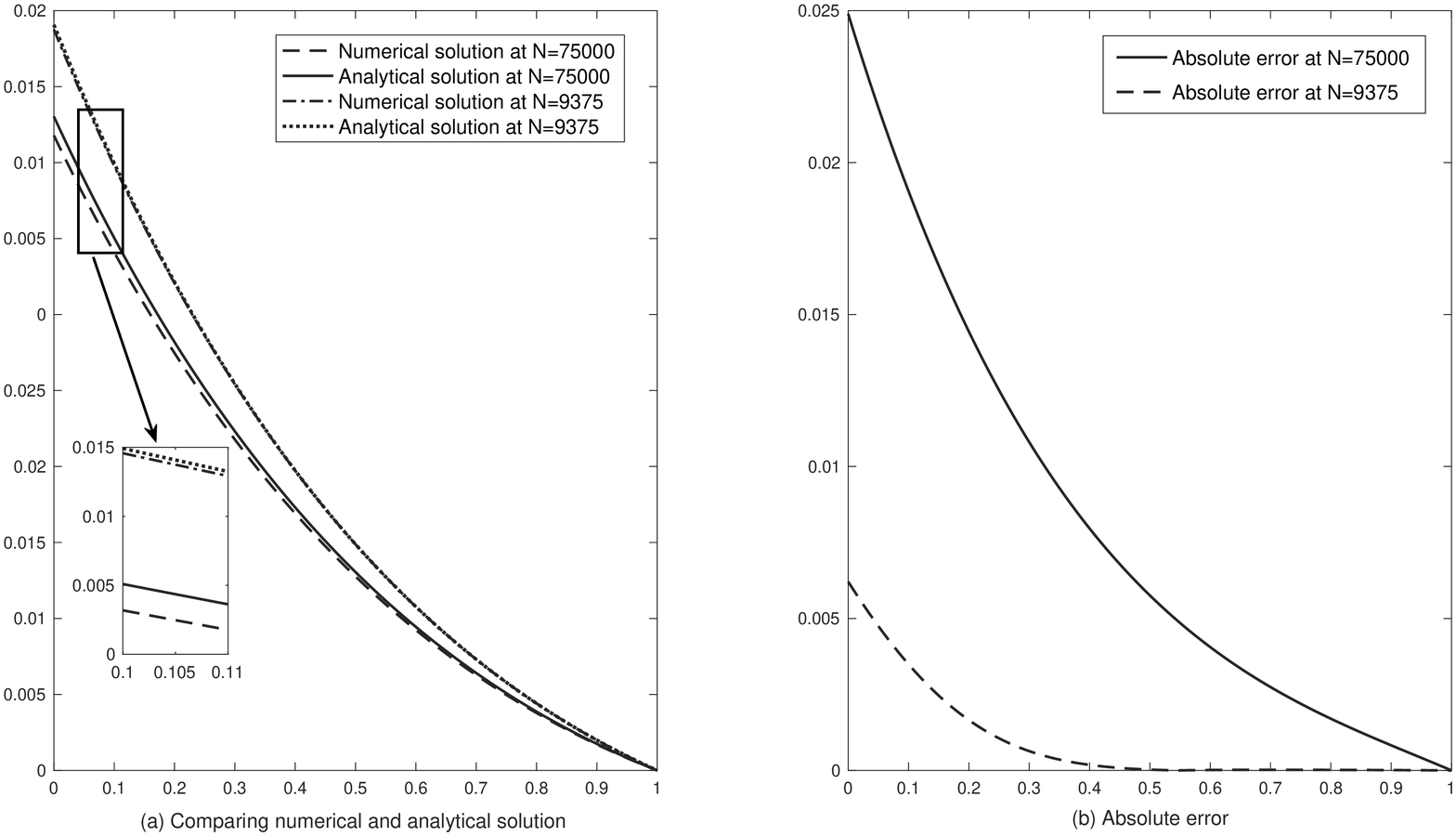}
	\caption{(a) The numerical solution and the analytical solution to (\ref{mvd-1})--(\ref{mvd-2}) with $N=9375,\ 75000$ with $d,B,g,u_0$ are given in Example \ref{eg02_lnr}; (b) The absolute difference between $u(N\Delta t, x_i)$ and $U^{N\Delta t}_{x_i}$ with $N=9375,\ 75000$.}
	\label{eg2_linear}
	\includegraphics[width=1.1\textwidth]{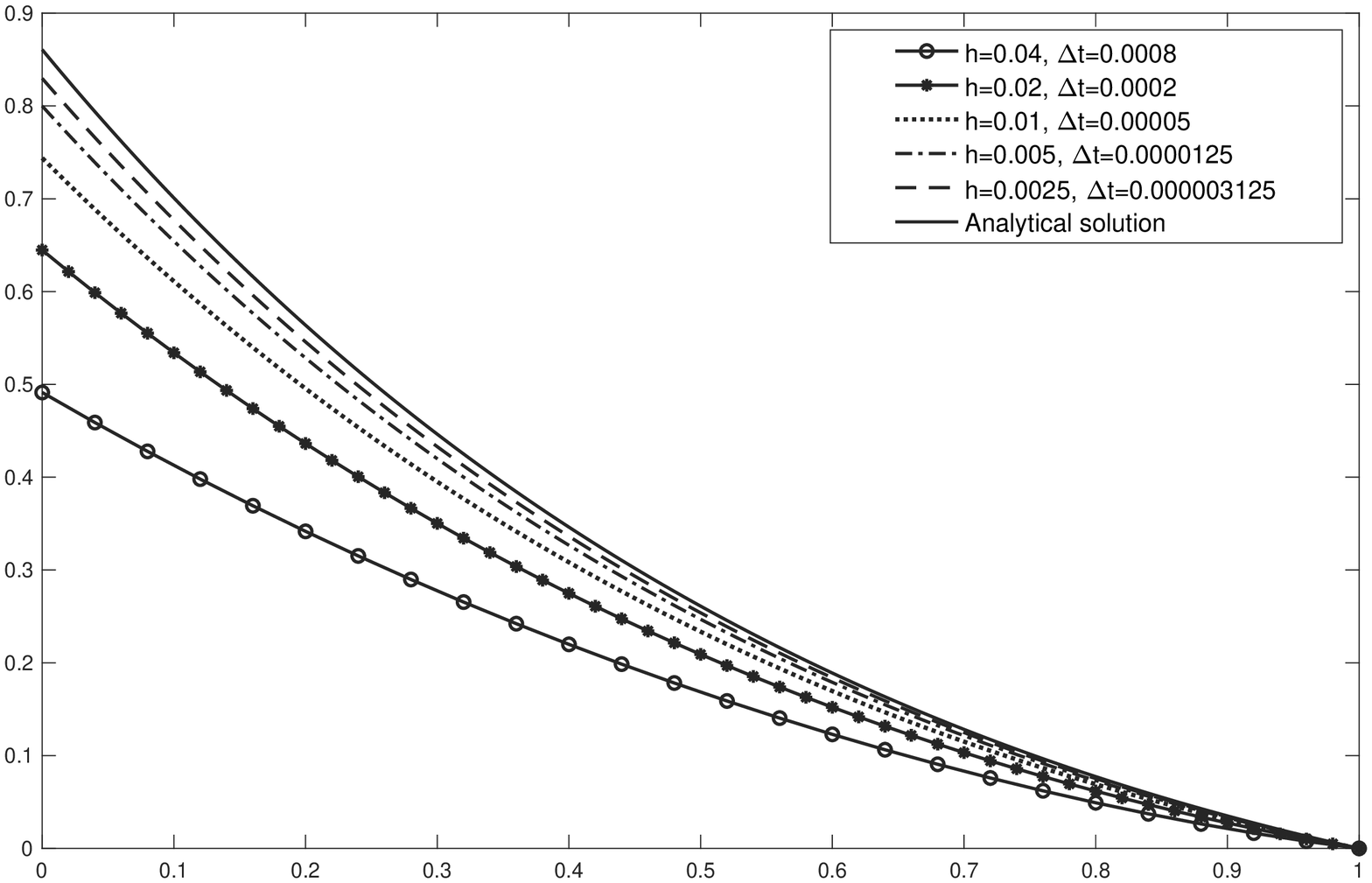}
	\caption{Numerical solution with different values of $h,\ \Delta t$ and the corresponding analytical solution to (\ref{mvd-1})--(\ref{mvd-2})  at $t=0.15$ with $d,B,g,u_0$ given in Example \ref{eg02_lnr}.}
	\label{eg2_linear_comp}
\end{figure}

\begin{example} \label{eg03_lnr}
	Here we take the mortality rate, the fertility rate and the initial data to be\medskip\\
	\centerline{$\displaystyle d(x,S)= \frac{1-e^{-\ad}-\ad e^{-\ad}}{S} + \frac{ e^{-x}-2e^{-\ad}}{2(e^{-x}-e^{-\ad})},\ B(x)=1+\frac{\ad e^{-\ad}}{1-e^{-\ad}-\ad e^{-\ad}}, $}\smallskip\\
	\[ u_0(x)=\frac{e^{-x}-e^{-\ad}}{2},\ \epsilon = 0.5,\ g(x)=x.\]
\end{example}
\noindent In this example the solution to (\ref{mvd-1})--(\ref{mvd-2}) is given by $$u(t,x)=\frac{1}{1+e^{-t}}(e^{-x}-e^{-\ad}),\ \ t>0,\ x\in (0,\ad).$$ In Figure  \ref{eg3_linear}(a), we present the analytical solution and the numerical solution when $N=20000$. In Figure \ref{eg3_linear}(b), we have shown the absolute error. For these computations, we have taken $h=0.001,\ \Delta t=5 \rm{x}10^{-7},$ $\ad =2$. In this example, we observe that $\|g'\|_{\infty}\|B\|_{\infty}\ad>1$ and this type of nonlinearity does not satisfy the hypotheses of Theorem \ref{convergence_thm}. However, we still notice that the numerical solution obtained using our scheme is indeed a good approximation to the analytical solution. This example suggests that assumption (\ref{cvg_assumption}) is sufficient but not necessary for the convergence of the numerical scheme presented in this article.

\begin{example} \label{eg_b_function}
	{\rm  In this example, we assume that the vital rates $d,B,$ and the initial data $u_0$ are 
		given by\medskip\\
		\centerline{$\displaystyle d(x,S)=1+S,\ B(x)=\frac{e^{-5x}}{10}, \epsilon = 0.5, g(x)=\sqrt{1+x}, \ad=7,$}\medskip\\
		\centerline{$u_0(x)= f(x)*\eta_{0.1}(x),$}\medskip\\
		where $\eta_{\epsilon}(x)$  is the standard mollifier with support in $[-0.1,0.1]$, $f*\eta_{0.1}$ denote the convolution of $f$ with $\eta_{0.1}$, and
		\begin{equation*}
		f(x)= 
		\left\{ \begin{array}{ll}
		1, &  {\rm if}\  0\leq x\leq 6, \medskip \\
		(x-7)^2, & {\rm if}\ 6\leq x\leq 7.
		\end {array} \right.
		\end{equation*}
		We have chosen $u_0(.)$ such that it is smooth. The step sizes for the simulations are $\Delta t=0.0002,$ and $ h=0.02$. 
		The numerical solutions for different values of $t$ are shown in Figure \ref{eg01_num}. From Figure \ref{eg01_num}, we observe that the numerical solution for $t\geq 3$ are very close to each other giving us an impression that the solution to \eqref{mvd-1}--\eqref{mvd-2} is converging to the corresponding steady state solution.  
		
		\begin{figure}[h]
			\includegraphics[width=1.2\textwidth]{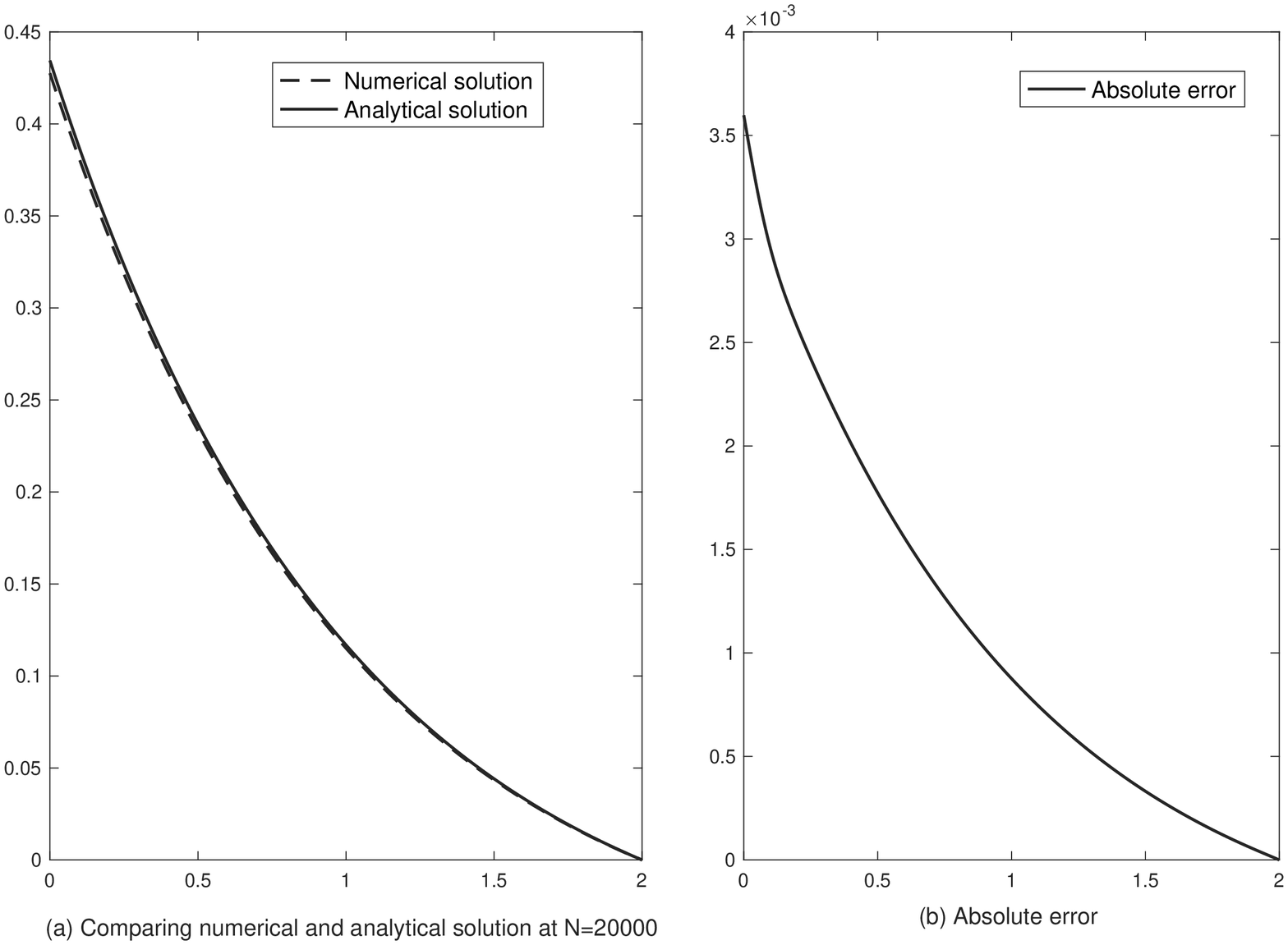}
			\caption{(a) The numerical solution and the analytical solution to (\ref{mvd-1})--(\ref{mvd-2}) with $N=20000$ with vital rates given in Example \ref{eg03_lnr}; (b) The absolute difference between $u(N\Delta t, x_i)$ and $U^{N\Delta t}_{x_i}$ with $N=20000$.}
			\label{eg3_linear}
			
			\includegraphics[width=1.1\textwidth]{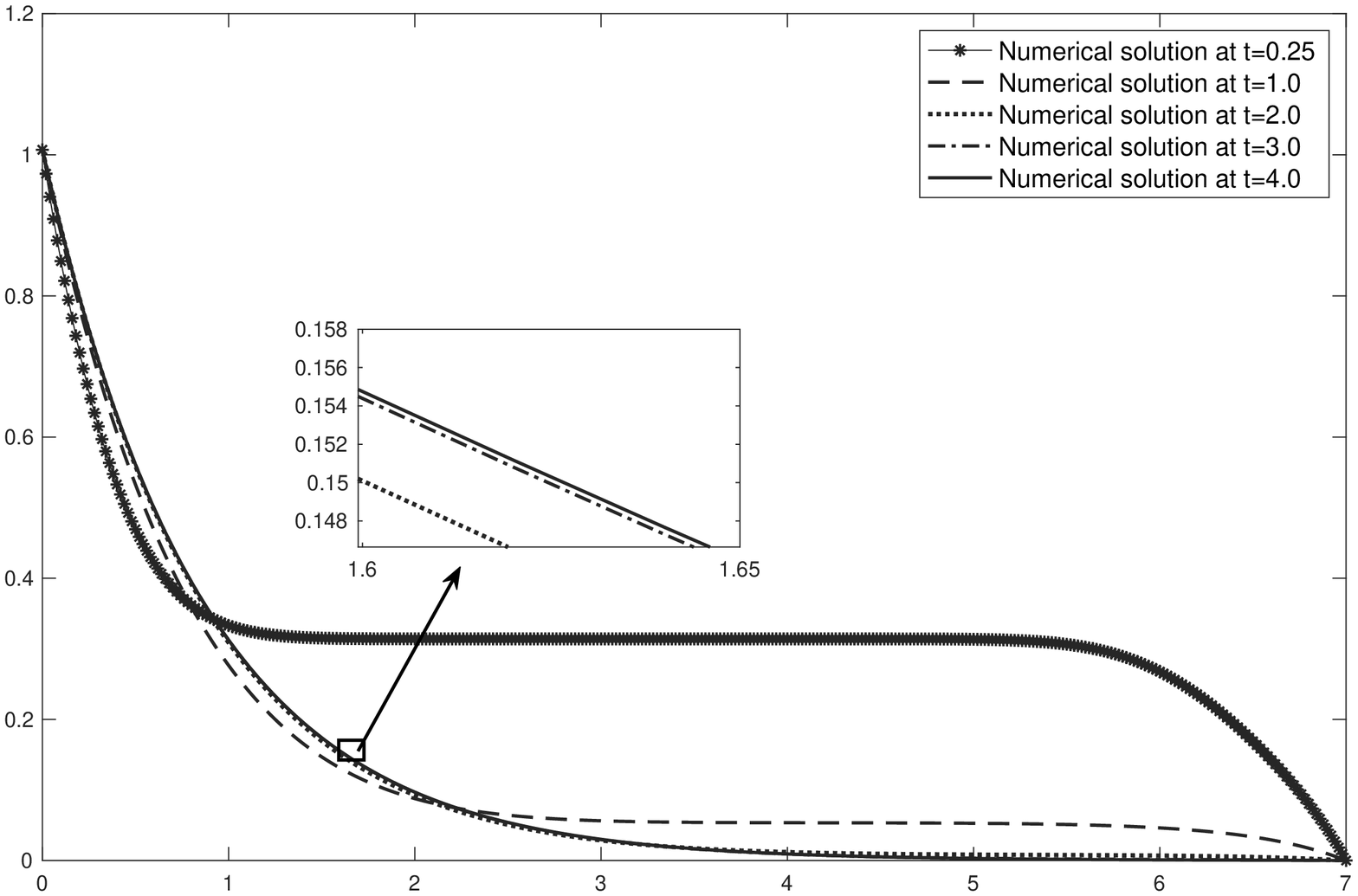}
			\caption{Numerical solution to (\ref{mvd-1})--(\ref{mvd-2}) with $N=5000$ with $d,B,g$ and $u_0$ given in Example \ref{eg_b_function}.}
			\label{eg01_num}
		\end{figure}
	}
\end{example}

\section{Conclusions}
In this paper, we have presented an implicit numerical scheme given in equation \eqref{num-sch}, for a nonlinear McKendrick--von Foerster equation with diffusion in age (MV-D). We have used the discretization based on the method of characteristics to $u_t+u_x$ and the finite difference method to the rest of the terms. We have also proved that the scheme we have considered is indeed convergent. As the approximation for the transport term is of the first order, the order of the proposed scheme turns out to be $O(h+\Delta t)$ (see equation (\ref{convergence_result})). Some numerical experiments are presented to revalidate the main theorem. From these simulations, numerical solutions are in a good  agreement with the corresponding analytical solutions. We have also presented an example where the hypothesis of Theorem \ref{convergence_thm} is not satisfied. Here also the  numerical solution obtained using our scheme provides a good approximation to the analytical solution. This shows that assumption (\ref{cvg_assumption}) in Theorem \ref{convergence_thm} can be weakened.

\section*{Acknowledgements} The authors would like to thank Prof. Jordi Ripoll for fruitful discussions. The first author would like to thank BITS Pilani, Hyderabad campus for providing financial support under Research Initiation Grant (BITS/GAU/RIG/2019/H0712).  The second author would like to acknowledge the support of DST-SERB, India, under MATRICS (MTR/2019/000848).

\end{document}